\documentclass[12pt]{amsart}
\usepackage[utf8]{inputenc}
\usepackage{array}
\usepackage{xcolor}
\usepackage{multirow}
\usepackage{longtable}
\usepackage{amsmath, amsthm, amssymb}
\usepackage{gensymb}
\usepackage{setspace}
\usepackage{amsrefs}
\usepackage{float}
\usepackage{comment}

\usepackage{tikz}
\usepackage{hyperref}
\usepackage{marginnote}
\usepackage{fullpage}

\usepackage{mathtools}

\DeclarePairedDelimiter\floor{\lfloor}{\rfloor}

\usepackage{graphics,graphicx}
\usepackage{pstricks,pst-node,pst-tree}
\usepackage{amsfonts}
\newcommand{\Z}{\mathbb{Z}}
\newcommand{\R}{\mathbb{R}}

\newcommand{\F}{\mathbb{F}}
\newcommand{\Span}{\text{Span}}
\newcommand{\wt}{\text{wt}}
\renewcommand{\P}[2]{P_{#1}^{(#2)}}

\theoremstyle{plain}
\newtheorem{theorem}{Theorem}[section] 
\newtheorem{proposition}[theorem]{Proposition}

\theoremstyle{definition}
\newtheorem{defn}[theorem]{Definition} 
\newtheorem{exmp}[theorem]{Example} 
\newtheorem{rmk}[theorem]{Remark}

\title{Classifying toric surface codes of dimension $7$}

\author[E.~Cairncross]{Emily Cairncross}
\address{Emily Cairncross \\ Oberlin College \\ Oberlin, Ohio \\ USA
\href{mailto:emily.cairncross@oberlin.edu}%
  {{\ttfamily\upshape emily.cairncross@oberlin.edu}}}
  
  \author[S.~Ford]{Stephanie Ford}
\address{Stephanie Ford\\ Texas A\&M \\ College Station, Texas \\ USA
\href{mailto:stephanieford@tamu.edu }%
  {{\ttfamily\upshape stephanieford@tamu.edu }}}
  
   \author[E.~Garcia]{Eli Garcia}
\address{Eli Garcia\\ M.I.T. \\ Cambridge, Massachusetts \\ USA
\href{mailto: etgarcia02@gmail.com  }%
  {{\ttfamily\upshape etgarcia02@gmail.com  }}}

\author[K.~Jabbusch]{Kelly Jabbusch}
\address{Kelly Jabbusch\\ Department of Mathematics \& Statistics\\ University
  of Michigan--Dearborn \\ Dearborn, Michigan
 \\ USA \\ 
  \href{mailto:jabbusch@umich.edu}%
  {{\ttfamily\upshape jabbusch@umich.edu}}}

\begin{document}

\maketitle

\begin{abstract}
    Toric surface codes are a class of error-correcting codes coming from a lattice polytope defining a two-dimensional toric variety. Previous authors have mostly completed classifications of these toric surface codes with dimension up to $k = 7.$ In this note, we correct an error in the classification of the $k=7$ case started in \cite{HLYZZ}, and disprove one of their conjectures.   
    
 \end{abstract}

\section{Introduction}

A specific toric code is constructed by first electing a finite field of order $q$ (where $q$ is a prime power) and a lattice convex polytope with $k$ lattice points (note that for the purposes of this paper, we only consider $2$-dimensional polytopes, or polygons, but toric codes \textit{can} be generated by higher-dimensional lattice convex polytopes). When a toric code is constructed from a polygon, which corresponds to a two-dimensional toric variety, we call it a toric surface code. A generator matrix can then be constructed involving both the elected field and polygon, and then the code consists of the set of linear combinations of the rows of the generator matrix. Given a toric code, we consider three parameters: \begin{itemize}

\vspace{.1 cm}

    \item The length of a codeword, which is $n = (q-1)^2$.
    
    \vspace{.2 cm}
    
    \item The dimension of the code, which is $k$.
    
    \vspace{.2 cm}
    
    \item The minimum Hamming distance $d$ of the code (Hamming distance counts the number of indices at which two codewords differ), which varies depending on the shape of the polygon. The greater the minimum distance, the more errors the code can correct. For example, if a codeword from a code with minimum distance 5 contained two errors, then that codeword would be closer to the intended codeword than any other codeword. So, those errors could be corrected.
  
\end{itemize}

\vspace{.1 cm}

The ideal code would have $n$ small (long codewords are hard to work with computationally) and $d$ large so that the code can correct as many errors as possible. For this reason, classifying toric codes based on their dimension is useful in finding patterns as to what shapes give better codes. To do this, one first finds all polygons that generate codes of dimension $k$ and computes their minimum distances. Then, various methods can be used to separate codes with the same minimum distance. This work was initiated by \cite{LSchwarz}, which classified toric surface codes with dimension $k\leq 5.$ The $k=5$ case was completed in \cite{dim5} and the $k=6$ case was done in \cite{LYZZComplete}. In \cite{HLYZZ}, the $k=7$ case was mostly completed. We achieved the results of \cite{HLYZZ} independently, and in this note we focus on correcting an error in the classification given by  \cite{HLYZZ} and disproving one of their conjectures. The strategy to classify toric surface codes with dimension $k=7$ follows the strategy given in the $k\leq 6$ cases.  One first determines the possible polygons with $7$ lattice points: 

\vspace{.1 cm}

\begin{theorem}\label{theorem:main1}  Every toric surface code with $k=7$ is monomially equivalent to a code generated by one of the $22$ polygons in Figure \ref{fig:k7}.
\end{theorem}

\vspace{.1 cm}

We will formally define what it means for two toric codes to be \emph{monomially equivalent} in Definition \ref{defn:mon equiv}, but the important thing is that monomially equivalent codes share values for all three parameters $n$, $k$, and $d$.

\begin{figure}  \label{fig:k7}  \caption{Lattice equivalence classes with $k=7$ lattice points}
    \centering
     \includegraphics[scale=1]{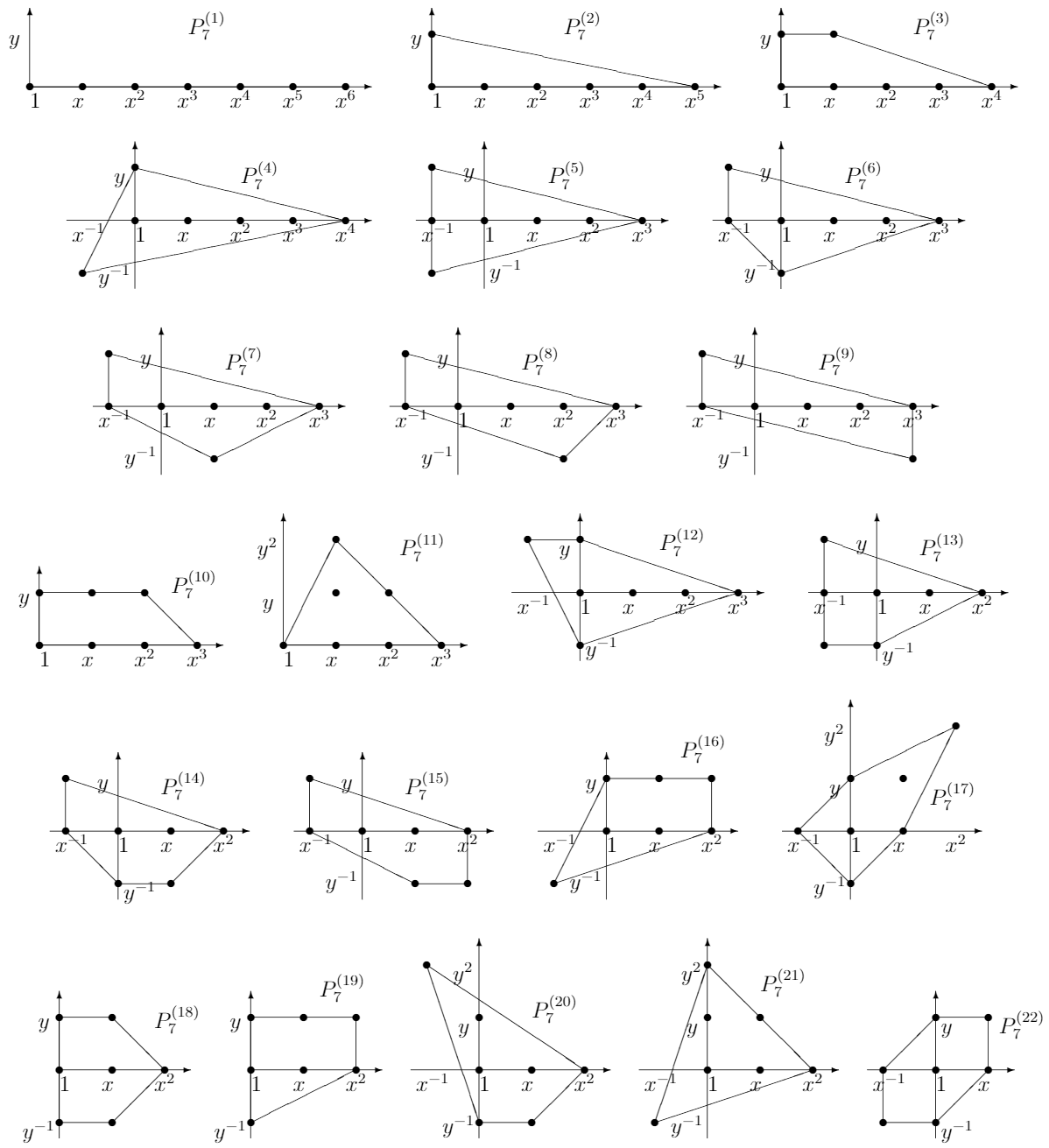}
\end{figure}

These $22$ polygons generate monomially inequivalent codes over $\F_q$ for most $q$, however, as in the case with $k=6$, we do find some cases in which two lattice inequivalent polygons generate monomially equivalent codes, as well as two open cases over $\F_8$.  More precisely:

\begin{theorem}\label{theorem:main2} These $22$ polygons generate monomially inequivalent codes over $\F_q$, for all $q$, with the following exceptions: \begin{enumerate}
    \item[(a)] $P_{7}^{(4)}$ and $P_{7}^{(8)}$ yield monomially equivalent  codes over $\F_7$, 
    \item[(b)] $P_{7}^{(5)}$ and $P_{7}^{(7)}$ yield monomially equivalent  codes over $\F_7$, 
    \item[(c)] $P_{7}^{(12)}$ and $P_{7}^{(13)}$ yield monomially equivalent  codes over $\F_7$, \end{enumerate}
The monomial equivalence of the following two cases remains open:
\begin{enumerate}
    \item[(d)] $P_{7}^{(4)}$ and $P_{7}^{(7)}$ over $\F_8$, 
    \item[(e)] $P_{7}^{(6)}$ and $P_{7}^{(5)}$ over $\F_8$. 
 \end{enumerate}
\end{theorem}

After the completion of this project, we learned that \cite{HLYZZ} had similar results. However, we answer one of their open cases and show that $P_{7}^{(10)}$ and $P_{7}^{(19)}$ yield monomially inequivalent codes over $\F_{29}$.  Many of our methods are similar, as we both extended the work of \cite{LYZZShort}, however we correct an error in their result regarding the minimum distances (Theorem \ref{theorem:mindist}).

The paper is organized as follows: in Section \ref{sec:prelim} we will first give an overview of definitions and previous results we will need to compute minimum distances.  In Section \ref{sec:proofs} we compute the minimum distances of the codes given by our $22$ polygons, correcting an error in the minimum distance formulas of \cite{HLYZZ}; and then complete the classification of the toric surface codes of dimension $k=7$, with the exception of the two cases over $\F_8$.  Finally, we end with a third remark  about toric surface codes of dimension $k=8$.

\section{Preliminary definitions and previous results}\label{sec:prelim}

Toric codes are a class of linear error-correcting codes introduced by Hansen in \cite{Hansen}. To construct such a code over the finite field $\F_q,$ we take a lattice convex polytope (i.e. the convex hull of a set of lattice points) $P \subset \square_{q-1} = [0, q-2]^m \subset \R^m.$ Then the toric code $C_P(\F_q)$ is given by the generator matrix defined by the following.

\begin{defn}
Let $\F_q$ be a finite field and $P \subset \square_{q-1} \subset \R^m$ be a lattice convex polytope. Write $\#(P) = |P \cap \Z^m|$ so that $\#(P)$ is the number of lattice points both on the boundary of and within the polytope. Then the \textit{toric code} $C_P(\F_q)$ is the linear code of block length $(q-1)^m$ given by the $\#(P) \times (q-1)^m$ generator matrix defined: $$G = (a^p),$$ for each $a \in (\F_q^{\ast})^m$ and each $p \in P \cap \Z^m,$ where $a^p = a_1^{p_1} \cdots a_m^{p_m}$ for $a = (a_1, \dots, a_m)$ and $p = (p_1, \dots, p_m).$

\end{defn}

\noindent Equivalently, $C_P(\F_q)$ can be defined as the image of an evaluation map. Let $$ \mathcal{L}(P) = \Span_{\F_q} \{x_1^{p_1}x_2^{p_2}\cdots x_m^{p_m} \colon p = (p_1, p_2, \ldots, p_m) \in P \cap \Z^m \},$$ then $C_P(\F_q)$ is the image of the map 
\begin{align*}
    \mathbf{\varepsilon} \colon  \mathcal{L}(P) & \to \F_q^{(q-1)^m} \\
                                f & \mapsto (f(a) \colon a \in (\F_q^{*})^m).
\end{align*}

\begin{exmp}
Let $q=5$ and $m=2$, and consider the polytope $T \subset \R^2$ with the $k=4$ lattice points $(0,0), (1,0), (0,1)$ and $(-1,-1)$, shown below.
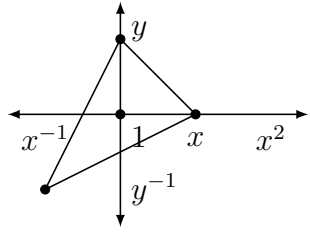
\begin{figure}[H] \label{fig:P4}
    \begin{center}
    \begin{tikzpicture}[scale=1]
    \draw [semithick, <->, >=latex] (-1.5,0) -- (2.5,0);
    \draw [semithick, <->, >=latex] (0,-1.5) -- (0,1.5);
    \node [below] at (-1,0) {$x^{-1}$};
    \node [right] at (0,1.1) {$y$};
    \node [below right] at (0,-0.02) {$1$};
    \node [right] at (0,-1) {$y^{-1}$};
    \node [below] at (1,-0.09) {$x$};
    \node [below] at (2,0) {$x^2$};
    \draw [semithick, black] (-1,-1) -- (0,1) -- (1,0) -- cycle;
    \draw [fill] (-1,-1) circle [radius=0.06];
    \draw [fill] (0,0) circle [radius=0.06];
    \draw [fill] (1,0) circle [radius=0.06];
    \draw [fill] (0,1) circle [radius=0.06];
    \end{tikzpicture}
    \end{center}
    \caption{The polytope $T$.} \label{figure:ExceptionalTriangle}
    \end{figure}
Then the toric code $C_{T}$ is the linear code given by the $4 \times 16$ generator matrix $G$, which we can calculate using every $a \in (\F_5^{\ast})^2$ and $p \in \{(-1,-1), \, (0,0), \, (1,0), \, (0,1)\}$. So let $g_{ij}$ be the element in $G$ such that $1 \leq i \leq 4$ and $1 \leq j \leq 16$. Each row of $G$ corresponds to a lattice point $p$, and each column corresponds to an element $a \in (\F_5^{\ast})^2$. For example, the first row of $G$ corresponds to the lattice point $(0,0)$, and the first column corresponds to $(1,1) \in (\F_5^{\ast})^2$. Thus, $g_{11} = 1^0 \cdot 1^0 = 1$. Since we take each element in $a \in (\F_5^{\ast})^2$ to the $0^{th}$ power, the first row of $G$ will be all 1's, i.e.
\begin{align*}
g_{1j} & =
\left( {\begin{array}{cccccccccccccccc}
    1 & 1 & 1 & 1 & 1 & 1 & 1 & 1 & 1 & 1 & 1 & 1 & 1 & 1 & 1 & 1 \\
    \end{array} } \right).
\end{align*}
The seventh column of $G$ corresponds to $(2,3) \in (\F_5^{\ast})^2$, and each row corresponds to (0,0), (1,0), (0,1), and (-1,-1), respectively. Thus we can calculate the seventh column of $G$.
\begin{align*}
    g_{17} = 2^0 \cdot 3^0 = 1 & \qquad g_{37} = 2^0 \cdot 3^1 = 3 \\
    g_{27} = 2^1 \cdot 3^0 = 2 & \qquad g_{47} = 2^{-1} \cdot 3^{-1} = 1
\end{align*}
Once we compute each element of $G$, we get

\begin{align*}
G = (a^p)
& =
\left( {\begin{array}{cccccccccccccccc}
    1 & 1 & 1 & 1 & 1 & 1 & 1 & 1 & 1 & 1 & 1 & 1 & 1 & 1 & 1 & 1 \\
    1 & 1 & 1 & 1 & 2 & 2 & 2 & 2 & 3 & 3 & 3 & 3 & 4 & 4 & 4 & 4 \\
    1 & 2 & 3 & 4 & 1 & 2 & 3 & 4 & 1 & 2 & 3 & 4 & 1 & 2 & 3 & 4 \\
    1 & 3 & 2 & 4 & 3 & 4 & 1 & 2 & 2 & 1 & 4 & 3 & 4 & 2 & 3 & 1 \\
    \end{array} } \right).
\end{align*}
Also note that $\mathcal{L}(T) = \Span_{\F_5}\{1, x, y, x^{-1}y^{-1}\}$.
\end{exmp}

\noindent From now on, we will generally omit the reference to $\F_q$ and write only $C_P.$ Realize that the weight of a codeword $w = \varepsilon(f) \in C_P$ is simply $$\wt(w) = (q-1)^m - Z(f),$$ where $Z(f)$ is the number of points in $(\F_q^*)^m$ at which $f$ vanishes. Hence, the minimum weight of $C_P$ is given by $$d(C_P) = (q-1)^m - \max_{0 \neq f \in L} Z(f).$$

We will focus on codes arising from toric surfaces, so $m=2$. We will classify the toric surface codes of dimension $7$ according to monomial equivalence, the precise definition is as follows:

\begin{defn}\label{defn:mon equiv} Let $C_1$ and $C_2$ be two codes of length $n$ and dimension $k$ over $\F_q$, and let $G_1$ and $G_2$ be generator matrices for $C_1$ and $C_2$, respectively.  $C_1$ and $C_2$ are \emph{monomially equivalent} if there is an invertible $n \times n$ diagonal matrix $\Delta$ and an $n \times n$ permutation matrix $\Pi$ such that $G_2=G_1 \Delta \Pi$. \end{defn}

In general it is difficult to check if two codes are monomially equivalent from the definition.  Little and Schwarz \cite{LSchwarz} give a more practical test for determining if two polytopes give the same code.  We first define lattice equivalence of polytopes:

\begin{defn} Two lattice convex polytopes $P_1$ and  $P_2$ in $\Z^m$ are \emph{lattice equivalent} if there exists a unimodular affine transformation $T : \R^m \to \R^m$ defined by $T(\mathbf{x}) = M\mathbf{x} + \lambda$ where $M \in \text{GL}(m, \Z)$ and $\lambda \in \Z^m$ such that $T(P_1) = P_2.$ \end{defn}

\begin{exmp} The two polygons labeled $P_1$ and $P_2$ below are lattice equivalent via $T(\mathbf{x}) = M\mathbf{x} + \lambda$ where $M= \begin{bmatrix} -1&0 \\ 0&1 \end{bmatrix}$ gives the reflection about the $y$-axis and $\lambda = \begin{bmatrix} 1 \\0 \end{bmatrix}$ gives the translation one unit to the right.

\begin{figure}[H]
    \begin{center}
    \begin{tikzpicture}[scale=1]
   \draw [semithick, <->, >=latex] (-1.5,0) -- (2.5,0);
    \draw [semithick, <->, >=latex] (0,-1.5) -- (0,1.5);
    \node [below] at (-1,0) {$x^{-1}$};
    \node [right] at (0,1.2) {$y$};
    \node [below right] at (0,-0.02) {$1$};
    \node [right] at (0.1,-1) {$y^{-1}$};
    \node [below] at (1,-0.09) {$x$};
    \node [below] at (2,0) {$x^2$};
    \node at (-1,1) {$P_1$};

    \draw [semithick, black] (-1,-1) -- (0,-1) -- (2,0) -- (1,1) -- (0,1)-- cycle;
    \draw [fill] (-1,-1) circle [radius=0.06];
    \draw [fill] (0,0) circle [radius=0.06];
    \draw [fill] (1,0) circle [radius=0.06];
    \draw [fill] (0,1) circle [radius=0.06];
    \draw [fill] (1,1) circle [radius=0.06];
    \draw [fill] (0,-1) circle [radius=0.06];
    \draw [fill] (2,0) circle [radius=0.06];

    \end{tikzpicture}    \hspace{.5in}
\begin{tikzpicture}[scale=1]
   \draw [semithick, <->, >=latex] (-1.5,0) -- (2.5,0);
    \draw [semithick, <->, >=latex] (0,-1.5) -- (0,1.5);
    \node [below] at (-1,0) {$x^{-1}$};
    \node [right] at (0,1.2) {$y$};
    \node [below right] at (0,-0.02) {$1$};
    \node [right] at (0,-1.1) {$y^{-1}$};
    \node [below] at (1,-0.09) {$x$};
    \node [below] at (2,0.05) {$x^2$};
    \node at (-1,1) {$P_2$};
    \draw [semithick, black] (-1,0) -- (1,-1) -- (2,-1) -- (1,1) -- (0,1)-- cycle;
    \draw [fill] (-1,0) circle [radius=0.06];
    \draw [fill] (0,0) circle [radius=0.06];
    \draw [fill] (1,0) circle [radius=0.06];
    \draw [fill] (0,1) circle [radius=0.06];
    \draw [fill] (1,1) circle [radius=0.06];
    \draw [fill] (1,-1) circle [radius=0.06];
    \draw [fill] (2,-1) circle [radius=0.06];

    \end{tikzpicture}
    \end{center}
    \caption{Lattice equivalent polygons via reflection and translation.} \label{figure:LatticeEquiv}
\end{figure}
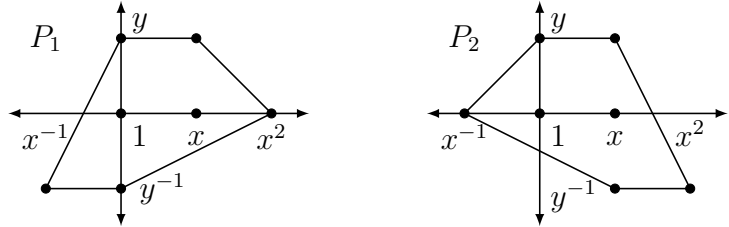
\end{exmp}

 \begin{exmp}
The two polygons $Q$ and $P_6^{(3)}$ below are also lattice equivalent via $T(\mathbf{x}) = M\mathbf{x} + \lambda$ where $M= \begin{bmatrix} 1&-1 \\ 0&1 \end{bmatrix}$ gives the shear to the left and $\lambda = \begin{bmatrix} 0 \\0 \end{bmatrix}$ (i.e., no translation).
 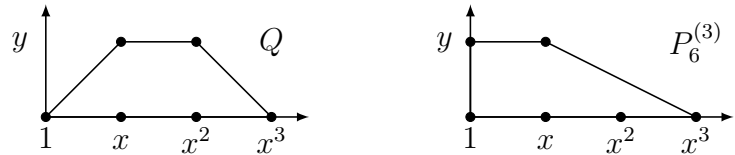
\begin{figure}[H]
    \begin{center}
    \begin{tikzpicture}[scale=1]
   \draw [semithick, ->, >=latex] (0,0) -- (3.5,0);
    \draw [semithick, ->, >=latex] (0,0) -- (0,1.5);
    \node [left] at (-0.09,1) {$y$};
    \node [below] at (0,0) {$1$};
    \node [below] at (1,-0.1) {$x$};
    \node [below] at (2,0) {$x^2$};
    \node [below] at (3,0) {$x^3$};
    \node at (3,1) {$Q$};

    \draw [semithick, black] (0,0) -- (1,0) -- (2,0) -- (3,0) -- (2,1) -- (1,1) -- cycle;
    \draw [fill] (0,0) circle [radius=0.06];
    \draw [fill] (1,0) circle [radius=0.06];
    \draw [fill] (2,0) circle [radius=0.06];
    \draw [fill] (3,0) circle [radius=0.06];
    \draw [fill] (1,1) circle [radius=0.06];
    \draw [fill] (2,1) circle [radius=0.06];

    \end{tikzpicture} \hspace{.5in}
     \begin{tikzpicture}[scale=1]
   \draw [semithick, ->, >=latex] (0,0) -- (3.5,0);
    \draw [semithick, ->, >=latex] (0,0) -- (0,1.5);
    \node [left] at (-0.09,1) {$y$};
    \node [below] at (0,0) {$1$};
    \node [below] at (1,-0.1) {$x$};
    \node [below] at (2,0) {$x^2$};
    \node [below] at (3,0) {$x^3$};
    \node at (3,1) {$P_6^{(3)}$};

    \draw [semithick, black] (0,0) -- (1,0) -- (2,0) -- (3,0) -- (1,1) -- (0,1) -- cycle;
    \draw [fill] (0,0) circle [radius=0.06];
    \draw [fill] (1,0) circle [radius=0.06];
    \draw [fill] (2,0) circle [radius=0.06];
    \draw [fill] (3,0) circle [radius=0.06];
    \draw [fill] (1,1) circle [radius=0.06];
    \draw [fill] (0,1) circle [radius=0.06];

    \end{tikzpicture}
    \end{center}
    \caption{Lattice equivalent polygons via shear.} \label{figure:LatticeEquiv2}
\end{figure}
 \end{exmp}

\begin{theorem}[\cite{LSchwarz}]  If two polytopes $P_1$ and $P_2$ are lattice equivalent, then the toric codes $C_{P_1}$ and $C_{P_2}$ are monomially equivalent.  \end{theorem}

Luo, Yau, Zhang, and Zuo \cite{LYZZShort} (and for more details \cite{LYZZComplete}) classified toric surface codes of dimension $k=6$, and found that for small $q$, it was possible that two polytopes could be lattice inequivalent but still yield monomially equivalent codes (a phenomena which didn't occur for $k<6$).  

\begin{theorem}
    [Luo, Yau, Zhang, and Zuo \cite{LYZZShort}] Every toric surface code with $k = 6$, is monomially equivalent to one of $14$ polygons, denoted by $C_{P_6^{(i)}}$ for $1\leq i \leq 14$.  Furthermore, $C_{P_6^{(i)}}$ and $C_{P_6^{(j)}}$ are not monomially equivalent over $\F_q$ for all $q \geq 7,$ except that
    \begin{enumerate}
        \item $C_{P_6^{(5)}}$ and $C_{P_6^{(6)}}$ over $\F_7$ are monomially equivalent;
        \item the monomial equivalence of $C_{P_6^{(4)}}$ and $C_{P_6^{(5)}}$ over $\F_8$ remains open.
    \end{enumerate}
\end{theorem}

\subsection{Results to compute the minimum distance}

 Recall the Minkowski sum of two polytopes $P$ and $Q$ is the pairwise sum of points in $P$ and $Q:$ $P+Q = \{x+y \, |\, x\in P, y \in Q\}$. 

\begin{defn} Let $P$ be a lattice polytope with Minkowski decomposition $P = P_1 + \cdots + P_l$, where each $P_i$ has positive dimension.  The \emph{Minkowski length of $P$}, denoted $l(P)$, is the largest number of summands in such a decomposition.  The \emph{full Minkowski length of $P$} is the maximum of the Minkowski lengths of all subpolytopes in $P$: $L(P) := \max\{ l(Q) \, | \, Q \subset P\}$.  \end{defn}

Soprunov and Soprunova give the following bound on the minimum distance of a code, based on the full Minkowski length $L$.

\begin{theorem}[\cite{SS}]\label{theorem:SSboundExc} Let $P \subset \square_{q-1}$ be a lattice polygon with area $A$ and full Minkowski length $L.$ For $q \geq \max \left\{ 23, \big(c + \sqrt{c^2 + {5}/{2}}\big)^2 \right\}$, where $c = {A}/{2}-L +{9}/{4}$, the minimum distance of the toric surface code $C_P$ satisfies
\[d(C_p) \geq (q-1)^2-L(q-1)-2\sqrt{q} + 1.\]
\end{theorem}

\begin{defn} An \emph{exceptional triangle} is a lattice polygon that has exactly three lattice points on the boundary and exactly one lattice point in the interior. 
Note that any exceptional triangle will be lattice equivalent to the polygon $T$ given in Figure \ref{figure:ExceptionalTriangle}. \end{defn}

\begin{exmp}  
Below is a polygon which is the Minkowski sum of the exceptional triangle $T$ and a unit line segment.  

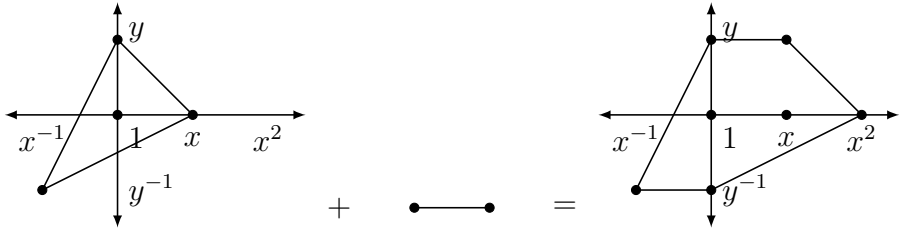
\begin{figure}[H]
    \begin{center}
    
    \begin{tikzpicture}[scale=1]
    \draw [semithick, <->, >=latex] (-1.5,0) -- (2.5,0);
    \draw [semithick, <->, >=latex] (0,-1.5) -- (0,1.5);
    \node [below] at (-1,0) {$x^{-1}$};
    \node [right] at (0,1.1) {$y$};
    \node [below right] at (0,-0.02) {$1$};
    \node [right] at (0,-1) {$y^{-1}$};
    \node [below] at (1,-0.09) {$x$};
    \node [below] at (2,0) {$x^2$};
    \draw [semithick, black] (-1,-1) -- (0,1) -- (1,0) -- cycle;
    \draw [fill] (-1,-1) circle [radius=0.06];
    \draw [fill] (0,0) circle [radius=0.06];
    \draw [fill] (1,0) circle [radius=0.06];
    \draw [fill] (0,1) circle [radius=0.06]; 
    \end{tikzpicture}  \begin{tikzpicture}[scale=1]
    \node at  (-1,2) {$+$};
    \node at  (2,2) {$=$};

    \draw [semithick, black] (0,2) -- (1,2)  -- cycle;
    \draw [fill] (0,2) circle [radius=0.06];
    \draw [fill] (1,2) circle [radius=0.06];
    \end{tikzpicture}  \begin{tikzpicture}[scale=1]
   \draw [semithick, <->, >=latex] (-1.5,0) -- (2.5,0);
    \draw [semithick, <->, >=latex] (0,-1.5) -- (0,1.5);
    \node [below] at (-1,0) {$x^{-1}$};
    \node [right] at (0,1.1) {$y$};
    \node [below right] at (0,-0.02) {$1$};
    \node [right] at (0,-1) {$y^{-1}$};
    \node [below] at (1,-0.09) {$x$};
    \node [below] at (2,0) {$x^2$};

    \draw [semithick, black] (-1,-1) -- (0,-1) -- (2,0) -- (1,1) -- (0,1)-- cycle;
    \draw [fill] (-1,-1) circle [radius=0.06];
    \draw [fill] (0,0) circle [radius=0.06];
    \draw [fill] (1,0) circle [radius=0.06];
    \draw [fill] (0,1) circle [radius=0.06];
    \draw [fill] (1,1) circle [radius=0.06];
    \draw [fill] (0,-1) circle [radius=0.06];
    \draw [fill] (2,0) circle [radius=0.06];

    \end{tikzpicture}
    \end{center}
    \caption{The Minkowski sum  of the exceptional triangle $T$ with a unit segment} \label{figure:MinkowskiSum}
    \end{figure}
\end{exmp}

Soprunov and Soprunova improved the bound of Theorem \ref{theorem:SSboundExc} in the case that the Minkowski decomposition of $P$ does not contain an exceptional triangle as one of the summands.  To state this result, recall that the Newton polytope of a polynomial $f$ is the convex hull of the exponent vectors of the monomials appearing in $f$.  For example the Newton polytope, $P_f$, of $f=ax + by+cx^{-1}y^{-1},$ where $a,b,c \in \F_q^*$, is the convex hull of three points: $P_f = \text{Conv}((1,0), (0,1), (-1,-1))$, which is the exceptional triangle $T$ depicted in Figure \ref{figure:ExceptionalTriangle}.  

\begin{theorem}[\cite{SS}] \label{theorem:SSboundNOExc} Let $P \subset \square_{q-1}$ be a lattice polygon with area $A$ and full Minkowski length $L.$ If  for every $f \in \mathcal{L}(P)$, there is no factorization  $f = f_1 \cdots f_L$, where the Newton polygon of one of the factors is an exceptional triangle, then  for $q \geq \max(37, \big(c + \sqrt{c^2 + 5/2}\big)^2)$, where $c = {A}/{2}-L +{11}/{4}$, the minimum distance of the toric surface code $C_P$ satisfies
\[d(C_p) \geq (q-1)^2-L(q-1).\]
\end{theorem}

Finally we record a result of \cite{SS} which gives a bound on the number of zeros of an absolutely irreducible polynomial $f$, denoted by $Z(f)$, which depends on $q$ and the number of interior points and primitive edges of the Newton polygon of $f$.  A primitive edge of a polygon is an edge whose only lattice points are the endpoints. For example, the exceptional triangle depicted above in Figure \ref{figure:ExceptionalTriangle} has three primitive edges.

\begin{theorem}[\cite{SS}]\label{theorem:SSZero}
    Let $f$ be absolutely irreducible with Newton polygon $P_f.$ Then $$ Z(f) \leq q + 1 + \floor{2I(P_f)\sqrt{q}} - B'(P_f), $$ where $I(P_f)$ is the number of interior lattice points and $B'(P_f)$ is the number of primitive edges of $P_f.$
\end{theorem}

\section{Toric surface codes of dimension $7$}\label{sec:proofs}

Using the lattice equivalence of polygons with $k=6$ lattice points, one can construct the $22$ equivalence classes of polygons with $k=7$ lattice points by finding all possible ways to add an extra point.  This is the analogous process that was employed in  \cite{LYZZShort} to construct the fourteen polygons with $6$ lattice points from those with $5$ lattice points; see \cite[Theorem 1.1]{HLYZZ} for a sketch of the construction for $k=7$.  Following the notation of \cite{LSchwarz}, \cite{dim5} and \cite{LYZZShort}, the polygons are denoted by $P_{k}^{(i)}$, where $k$ is the number of lattice points and $i$ denotes the equivalence class.

\begin{theorem}\label{theorem:k7} \cite[Theorem 1.1]{HLYZZ} Every toric surface code with $k=7$ is monomially equivalent to a code generated by one of the $22$ polygons in Figure \ref{fig:k7}. \end{theorem}

\bigskip

To determine whether the $22$ polygons yield monomially inequivalent codes, we first compute (or bound) the minimum distance of each code.   This is also the first step taken in  \cite[Prop 3.5]{HLYZZ}, however they make an error in computing the minimum distance of the codes arising from the polygons $P_7^{(i)}$, for $i=16,18$ and $19$, which we correct.  Additionally we verify that the minimum distance of $C_{P_7^{(22)}}$ is $(q-2)(q-3)$, whereas \cite{HLYZZ} had only bounded the minimum distance from below by $(q-2)(q-3)$.

\begin{theorem}\label{theorem:mindist}
    The polygons with $k=7$ lattice points generate codes with minimum distances given by the formulas in Table \ref{table:2} below for sufficiently large $q$.
\end{theorem}

\begin{center}
\begin{longtable}[c]{|| c | c | c ||}
\caption{Minimum distances  \label{table:2}}\\

    \hline
    lattice equivalence class & minimum distance formula & bound on $q$ \\
    \hline\hline
    \endfirsthead

    \hline
    lattice equivalence class & minimum distance formula & bound on $q$ \\
    \hline\hline
    \endhead
    
    \hline
    \endfoot
         
    \hline\hline
    \endlastfoot
    
    $P_7^{(1)}$ & $(q-1)(q-7)$ & all $q$ \\
    \hline
    $P_7^{(2)}$ & $(q-1)(q-6)$ & all $q$ \\
    \hline
    $P_7^{(3)}$ & $(q-1)(q-5)$ & all $q$ \\
    $P_7^{(4-9)}$ & $(q-1)(q-5)$ & $q \geq 37$ \\
    \hline
    $P_7^{(10-11)}$ & $(q-1)(q-4)$ & all $q$ \\
    $P_7^{(12-15)}$ & $(q-1)(q-4)$ & $q \geq 37$ \\
    \hline
    $P_7^{(18,19, 22)}$ & $(q-2)(q-3)$ & $q \geq 5$ \\
    $P_7^{(16)}$ & $(q-2)(q-3)$ & $q \geq 9$ \\
    \hline
    $P_7^{(17)}$ & $(q-1)(q-3) \geq d > (q-2)(q-3)$ & $q \geq 23$ \\
    \hline
    $P_7^{(20,21)}$ & $(q-1)(q-3)$ & $q \geq 37$ \\
\end{longtable}
\end{center}

\begin{proof}
For the minimum distances of the codes coming  from $P_7^{(i)}$ for $1 \leq i\leq 15$ and $i=17, 20, 21$ we refer the reader to  \cite[Prop 3.5]{HLYZZ}.

   First, we consider codes coming from $P_7^{(i)}$, $i=18,19,22$.  We first note that we can find a polynomial in $\mathcal{L}(P_7^{(i)})$, $i=18,19,22$, with $3(q-1)-2$ zeros.  Let $a,b,c \in \F_q^*$,  the polynomial $(x-a)(y-b)(y^{-1}-c) \in \mathcal{L}(P_7^{(18)})$ has $3(q-1)-2$ zeros if $b\neq c^{-1}$, the polynomial $(x-a)(x-b)(y-c) \in \mathcal{L}(P_7^{(19)})$ has $3(q-1)-2$ zeros if $a\neq b$, and the polynomial $(x-a)(y-bx^{-1})(y^{-1}-c) \in \mathcal{L}(P_7^{(22)})$ has $3(q-1)-2$ zeros if and only if $a=bc$ (else it will have $3(q-1)-3$ zeros).  
   
   We next show that for $q \geq 5$, any $f \in \mathcal{L}(P_7^{(i)})$, $i=18,19,22$ will have at most $3(q-1)-2$ zeros.  This then will give us that $ d(C_{P_7^{(i)}}) = (q-1)^2 - (3(q-1)- 2) = (q-2)(q-3)$.

Note that each $P_7^{(i)}$, $i=18,19,22$, has full Minkowski length $L=3$, and any maximal decomposition in $P_7^{(i)}$ will be a Minkowski sum of three primitive edges, which implies that every polynomial with the largest number of absolutely irreducible factors (three) will have at most $3(q-1) - 2$ zeros in $(\F_{q}^*)^2$.

Next consider the case where $f$ is absolutely irreducible, with Newton polytope $P_f$.  Since each $P_7^{(i)}$, $i=18,19,22$, has one interior point and $P_f \subseteq P_7^{(i)}$, the number of interior points of $P_f$, $I(P_f)$, is at most $1$.  Then by Theorem \ref{theorem:SSZero}, 
$$Z(f) \leq q+1 + \lfloor 2I(P_f)\sqrt{q}\rfloor - B'(P_f) \leq q+1 +\lfloor 2\sqrt{q} \rfloor.$$
For $q \geq 5,$ $ q+1 +\lfloor 2\sqrt{q} \rfloor \leq 3(q-1)-2.$

Finally consider the case where $f$ factors into two absolutely irreducible polynomials.  One can apply Propositon 2.4 from \cite{SS}, to get that for $q \geq 41$, $Z(f) \leq 3(q-1)-2$, but following the proof of \cite[Prop 2.4]{SS}, we can bring the bound on $q$ to $q \geq 5$.  Let $f=f_1f_2,$ and let $P_i$ be the Newton polygon of $f_i$, so that $P_f = P_1 + P_2$.  Then, as in the proof of  \cite[Prop 2.4]{SS}, $L(P_1) \leq 2$ and $L(P_2) =1$.  

\begin{itemize}
\item If $L(P_1)$ and $L(P_2)$ are both one, then $P_1$ and $P_2$ are both strongly indecomposable triangles or lattice segments.  Note that neither can be the exceptional triangle, since the Minkowski sum of the exceptional triangle with a line segment or simplex is not contained in $P_7^{(i)}$.  So if $L(P_1) = L(P_2) =1$, then 
$$ Z(f) \leq 2(q-1) \leq 3(q-1) - 2. $$

\item If $L(P_1)=2$ and $L(P_2)=1$, then as before $P_2$ is either the two-simplex or a lattice segment.  Since the Minkowski sum of $P_1$ and $P_2$ must be contained in $P_7^{(i)}$, for $i=18,19,22$, the only possibility for $P_1$ is the unit square, which has no interior points. 
Applying Theorem \ref{theorem:SSZero} to $f_1$ gives $Z(f_1) \leq q+1 +\lfloor 2\cdot 0\sqrt{q} \rfloor - B'(P_1) \leq q+1,$ and $Z(f_2) \leq q-1$, thus
$$ Z(f)  \leq q+1 + q-1 = 2q  $$
For $q \geq 5$, this is smaller than $3(q-1)-2$.
\end{itemize}
   
For the code coming from $P_7^{(16)}$ the minimum distance is computed similarly to the code above coming from $P_7^{(i)}$, $i=18,19,22$.  The difference is that $P_7^{(16)}$ has two interior points, whereas $P_7^{(i)}$, $i=18,19,22$, have only one.  This will change our bound on $q$ slightly to determine that $ d(C_{P_7^{(16)}}) = (q-1)^2 - (3(q-1)- 2) = (q-2)(q-3)$, for $q \geq 9.$  First note the polynomial $f=(x-a)(x-b)(y-c) \in \mathcal{L}(P_7^{(16)})$, with $a,b,c \in \F_q^*$, $a\neq b$ has $3(q-1)-2$ zeros.  As above we'll show any other polynomial $f\in \mathcal{L}(P_7^{(16)})$ has $Z(f) \leq 3(q-1)-2$ for $q \geq 9.$  As in the case of $P_7^{(i)}$, $i=18,19,22$, every polynomial with the largest number of absolutely irreducible factors (three) will have at most $3(q-1) - 2$ zeros in $(\F_{q}^*)^2$. If $f$ is absolutely irreducible, with Newton polytope $P_f$, then the number of interior points of $P_f$, $I(P_f)$, is at most $2$.  Then by Theorem \ref{theorem:SSZero}, 
$$Z(f) \leq q+1 + \lfloor 2I(P_f)\sqrt{q}\rfloor - B'(P_f) \leq q+1 +\lfloor 4\sqrt{q} \rfloor \leq 3(q-1)-2, \text{ for } q \geq 9.  $$
In the case where $f$ factors into two absolutely irreducible polynomials, $f=f_1f_2$, we have that $L(P_1) \leq 2$ and $L(P_2) =1.$ If $L(P_1)=L(P_2)=1$, then as above, $Z(f) \leq 2(q-1) \leq 3(q-1)-2$. If $L(P_1)=2$ and $L(P_2)=1$, then as above $P_1$ must be the unit square, which has no interior points.  Applying Theorem \ref{theorem:SSZero} to $f_1$ gives 
$$Z(f) \leq Z(f_1) + Z(f_2) \leq q+1 +\lfloor 2\cdot 0\sqrt{q} \rfloor - B'(P_1) +q-1 \leq   q+1 + q-1 = 2q  $$
For $q \geq 5$, this is smaller than $3(q-1)-2$.

\end{proof}

\bigskip

If two codes have different minimum distances, we know that the codes are not monomially equivalent. Based on the previous proposition, we have that $P_7^{(1)}$ and $P_7^{(2)}$ give codes that are not monomially equivalent to any other. For all of the others though, there are many polygons that yield codes whose minimum distances coincide. We will look at these groups of polygons in turn and use finer invariants to distinguish the codes from each other.  We will focus on the number of codewords of particular weights.  Given a code $C_P$, denote by $n_1(C_P)$ the number of codewords of weight $(q-1)^2-(2q-2)$, $n_2(C_P)$ the number of codewords of weight $(q-1)^2-(2q-3)$, and $n_3(C_P)$ the number of codewords of weight $(q-1)^2-(3q-5)$. The general strategy is to analyze a group of polygons that yield codes with the same minimum distance and show that one of the above invariants differs.  In \cite{HLYZZ} a similar analysis is completed; they distinguish various codes using the invariants $n_1$ and $n_2$. Slightly more concise arguments can be given by considering $n_3$ in addition to $n_1$ and $n_2$, but we omit the proofs.  

\bigskip

\begin{proposition}\label{prop:prop1}

    \begin{enumerate}
        \item \cite[Prop 3.6, 3.7]{HLYZZ} $P_7^{(3-15)}$ all generate monomially inequivalent codes for $q>9.$
        \item The codes with minimum distance $(q-2)(q-3)$, generated by $P_7^{(16,18,19,22)}$, yield monomially inequivalent codes for  $q>9.$
        \item The codes with minimum distance $d$, $(q-1)(q-3) \geq d > (q-2)(q-3)$, generated by $P_7^{(17, 20,21)}$, yield monomially inequivalent codes for  $q>9.$
    \end{enumerate}
\end{proposition}

Note that statement (2) is similar to \cite[Prop 3.7]{HLYZZ}, but because we correctly compute the minimum distances of $C_{P_7^{(16,18,19)}},$ we are able to distinguish $C_{P_7^{(10)}}$ from $C_{P_7^{(19)}}.$ To distinguish $\P{7}{22}$ from $\P{7}{16}$, $\P{7}{18}$ and $\P{7}{19}$, one can first compute the invariant $n_3$.  More precisely,  note that polynomials of the form $d(y-a)(x-b)(x-c)$, with $b \neq c$, have exactly $3q-5$ zeros and are in $\mathcal{L}(\P{7}{i})$ for $i=16,19$.  Similarly, polynomials of the form $dy^{-1}(y-a)(y-b)(x-c)$, with $a \neq b$, have exactly $3q-5$ zeros and are in $\mathcal{L}(\P{7}{18})$.   Because there are $\binom{q-1}{2} (q-1)^2$ polynomials of this kind, there are at least as many words of weight $(q-1)^2-(3q-5)$ in $C_{\P{7}{16,18,19}}.$  In the proof that $C_{\P{7}{22}}$ has minimum distance $(q-1)^2 - (3q-5),$ Theorem \ref{theorem:mindist}, we show that for $q \geq 5,$ the only polynomials in  $\mathcal{L}(\P{7}{22})$ with $3q-5$ zeros are those of the form $d(x-a)(y-bx^{-1})(y^{-1}-c),$ where $a = bc$ and $a,b,c,d \in \F_q^*$. There are exactly $(q-1)^3$ of these polynomials. Since $(q-1)^3 < \binom{q-1}{2} (q-1)^2,$ this shows that $\P{7}{22}$ yields  a different code than those coming from $\P{7}{16}$, $\P{7}{18}$ or $\P{7}{19}$. As in the proof of  \cite[Prop 3.7]{HLYZZ}, one can then use $n_1$ and $n_2$ to distinguish $\P{7}{16}$, $\P{7}{18}$ and $\P{7}{19}$.

Statement (3) is an analogue of \cite[Prop 3.8]{HLYZZ}, which also considers the  code generated by $P_7^{(22)}.$ Because we computed the minimum distance of this code exactly, we have already distinguished it from the codes arising from $P_7^{(17)}, P_7^{(20)}$ and $P_7^{(21)}$.

Now we compile the results of the previous propositions to prove the main theorem, and use Sage to address monomial equivalence over fields of small $q$. 
Comparing with \cite[Theorem 1.2]{HLYZZ}, we have  added to their classification, by addressing the case of $C_{P_7^{(10)}}$  and $C_{P_7^{(19)}}$ over $\F_{29}$.  In this case, the Conjecture \cite[Conjecture 1.1]{HLYZZ} is false: $C_{P_7^{(10)}}$  and $C_{P_7^{(19)}}$ yield monomially inequivalent  codes over $\F_{29}$ (and in fact over any finite field), as the two codes have different minimum distances.    

\begin{theorem}\label{theorem:theorem2}
The $22$ polygons generate monomially inequivalent codes over $\F_q$, for all $q$, with the following exceptions: \begin{enumerate}
    \item[(a)] $P_{7}^{(4)}$ and $P_{7}^{(8)}$ yield monomially equivalent  codes over $\F_7$, 
    \item[(b)] $P_{7}^{(5)}$ and $P_{7}^{(7)}$ yield monomially equivalent  codes over $\F_7$, 
    \item[(c)] $P_{7}^{(12)}$ and $P_{7}^{(13)}$ yield monomially equivalent  codes over $\F_7$,
\end{enumerate} The monomial equivalence of the following two cases remains open: \begin{enumerate}
    \item[(d)] $P_{7}^{(4)}$ and $P_{7}^{(7)}$ over $\F_8$, 
    \item[(e)] $P_{7}^{(6)}$ and $P_{7}^{(5)}$ over $\F_8$. 
 \end{enumerate}
\end{theorem}

\begin{rmk}
In both the cases of $k=6$ and $k=7$ there are pairs of lattice inequivalent polytopes yielding monomially equivalent codes over $\F_7$, whereas the question of monomial equivalence over $\F_8$ remains open.  Computer checks using Sage and GAP can verify that $C_{P_6^{(4)}}$ and $C_{P_6^{(5)}}$ share the same enumerator polynomial  over $\F_8$, as do the pairs $C_{P_7^{(4)}}$ and $C_{P_7^{(7)}}$, and $C_{P_7^{(6)}}$ and  $C_{P_7^{(5)}}$.  While a shared enumerator polynomial does not guarantee monomial equivalence, it does provide compelling evidence that the pairs do yield monomially equivalent codes over $\F_8$.  For the pairs over $\F_7$, Joyner provides code to verify the monomial equivalence \cite{Joyner04}.  Attempts were made to extend this to codes over $\F_8$, but remained unsuccessful.  
\end{rmk}
       
\begin{rmk}
Employing a similar strategy as in Theorem \ref{theorem:k7} one can construct all the lattice inequivalent polygons with $8$ lattice points.  There are $42$ such polygons, so every toric surface code of dimension $k=8$, will be monomially equivalent to a code generated by one of the $42$ such polygons.
\end{rmk}

\section{Acknowledgements}
This research was conducted at the NSF REU Site (DMS-1659203) in Mathematical Analysis and Applications at the University of Michigan-Dearborn. We would like to thank the National Science Foundation, National Security Agency, University of Michigan-Dearborn (SURE 2019), and the University of Michigan-Ann Arbor for their support. Additionally, we  thank the other participants of the REU program for fruitful conversations on this topic, as well as the anonymous referee for their constructive feedback.

\bibliography{REUbib}

\end{document}